\documentclass[11pt]{amsart}

\usepackage[all]{xy}
\usepackage{mathrsfs}
\usepackage{amsmath, amsthm, amssymb, amscd, amsgen,amsxtra,amsfonts, amsbsy}
\allowdisplaybreaks
\usepackage[all]{xy}
\usepackage{verbatim}

\usepackage{tikz}
\usepackage{pgfplots}
\usepackage{tikz-cd}
\usepackage{etex}
\usepackage{calligra,mathrsfs}
\usepackage{color}
\definecolor{shadecolor}{gray}{0.875}
\definecolor{dblue}{rgb}{0,0,.6}
\usepackage[colorlinks=true, linkcolor=dblue, citecolor=dblue, filecolor = dblue, menucolor = dblue, urlcolor = dblue]{hyperref}
\usepackage{mathtools}
\usepackage{extarrows}
\usepackage{ifthen}
\usepackage{bm}

\usepackage{graphicx}
\usepackage{caption,rotating}
\usepackage{color}
\usepackage{subcaption}

\usepackage{enumitem}

\newcommand{\mathds}[1]{{\mathbb #1}}

\numberwithin{equation}{section}

\begin{document}
%
%
%
\theoremstyle{definition}
\newtheorem{Definition}{Definition}[section]
\newtheorem*{Definitionx}{Definition}
\newtheorem{Convention}{Definition}[section]
\newtheorem{Construction}[Definition]{Construction}
\newtheorem{Example}[Definition]{Example}
\newtheorem{Exercise}[Definition]{Exercise}
\newtheorem{Examples}[Definition]{Examples}
\newtheorem{Remark}[Definition]{Remark}
\newtheorem*{Remarkx}{Remark}
\newtheorem{Remarks}[Definition]{Remarks}
\newtheorem{Caution}[Definition]{Caution}
\newtheorem{Conjecture}[Definition]{Conjecture}
\newtheorem*{Conjecturex}{Conjecture}
\newtheorem{Question}[Definition]{Question}
\newtheorem*{Questionx}{Question}
\newtheorem*{Acknowledgements}{Acknowledgements}
\newtheorem*{Notation}{Notation}
\newtheorem*{Organization}{Organization}
\newtheorem*{Disclaimer}{Disclaimer}
\theoremstyle{plain}
\newtheorem{Theorem}[Definition]{Theorem}
\newtheorem*{Theoremx}{Theorem}

\newtheorem{Proposition}[Definition]{Proposition}
\newtheorem*{Propositionx}{Proposition}
\newtheorem{Lemma}[Definition]{Lemma}
\newtheorem{Claim}[Definition]{Claim}
\newtheorem*{ClaimA}{Claim A}
\newtheorem{Corollary}[Definition]{Corollary}
\newtheorem*{Corollaryx}{Corollary}
\newtheorem{Fact}[Definition]{Fact}
\newtheorem{Facts}[Definition]{Facts}
\newtheoremstyle{voiditstyle}{3pt}{3pt}{\itshape}{\parindent}%
{\bfseries}{.}{ }{\thmnote{#3}}%
\theoremstyle{voiditstyle}
\newtheorem*{VoidItalic}{}
\newtheoremstyle{voidromstyle}{3pt}{3pt}{\rm}{\parindent}%
{\bfseries}{.}{ }{\thmnote{#3}}%
\theoremstyle{voidromstyle}
\newtheorem*{VoidRoman}{}

\newenvironment{specialproof}[1][\proofname]{\noindent\textit{#1.} }{\qed\medskip}
\newcommand{\blowup}{\rule[-3mm]{0mm}{0mm}}
\newcommand{\cal}{\mathcal}
\newcommand{\Aff}{{\mathds{A}}}
\newcommand{\BB}{{\mathds{B}}}
\newcommand{\CC}{{\mathds{C}}}
\newcommand{\EE}{{\mathds{E}}}
\newcommand{\FF}{{\mathds{F}}}
\newcommand{\GG}{{\mathds{G}}}
\newcommand{\HH}{{\mathds{H}}}
\newcommand{\NN}{{\mathds{N}}}
\newcommand{\ZZ}{{\mathds{Z}}}
\newcommand{\PP}{{\mathds{P}}}
\newcommand{\QQ}{{\mathds{Q}}}
\newcommand{\RR}{{\mathds{R}}}
\newcommand{\BA}{{\mathds{A}}}
\newcommand{\Liea}{{\mathfrak a}}
\newcommand{\Lieb}{{\mathfrak b}}
\newcommand{\Lieg}{{\mathfrak g}}
\newcommand{\Liem}{{\mathfrak m}}
\newcommand{\ideala}{{\mathfrak a}}
\newcommand{\idealb}{{\mathfrak b}}
\newcommand{\idealg}{{\mathfrak g}}
\newcommand{\idealm}{{\mathfrak m}}
\newcommand{\idealp}{{\mathfrak p}}
\newcommand{\idealq}{{\mathfrak q}}
\newcommand{\idealI}{{\cal I}}
\newcommand{\lin}{\sim}
\newcommand{\num}{\equiv}
\newcommand{\dual}{\ast}
\newcommand{\iso}{\cong}
\newcommand{\homeo}{\approx}
\newcommand{\mm}{{\mathfrak m}}
\newcommand{\pp}{{\mathfrak p}}
\newcommand{\qq}{{\mathfrak q}}
\newcommand{\rr}{{\mathfrak r}}
\newcommand{\pP}{{\mathfrak P}}
\newcommand{\qQ}{{\mathfrak Q}}
\newcommand{\rR}{{\mathfrak R}}
\newcommand{\OO}{{\cal O}}
\newcommand{\numero}{{n$^{\rm o}\:$}}
\newcommand{\mf}[1]{\mathfrak{#1}}
\newcommand{\mc}[1]{\mathcal{#1}}
\newcommand{\into}{{\hookrightarrow}}
\newcommand{\onto}{{\twoheadrightarrow}}
\newcommand{\Spec}{{\rm Spec}\:}
\newcommand{\BigSpec}{{\rm\bf Spec}\:}
\newcommand{\Spf}{{\rm Spf}\:}
\newcommand{\Proj}{{\rm Proj}\:}
\newcommand{\Pic}{{\rm Pic }}
\newcommand{\MW}{{\rm MW }}
\newcommand{\Br}{{\rm Br}}
\newcommand{\NS}{{\rm NS}}
\newcommand{\Sym}{{\mathfrak S}}
\newcommand{\Aut}{{\rm Aut}}
\newcommand{\Autp}{{\rm Aut}^p}
\newcommand{\ord}{{\rm ord}}
\newcommand{\coker}{{\rm coker}\,}
\newcommand{\divisor}{{\rm div}}
\newcommand{\Def}{{\rm Def}}
\newcommand{\rank}{\mathop{\mathrm{rank}}\nolimits}
\newcommand{\Ext}{\mathop{\mathrm{Ext}}\nolimits}
\newcommand{\EXT}{\mathop{\mathscr{E}{\kern -2pt {xt}}}\nolimits}
\newcommand{\Hom}{\mathop{\mathrm{Hom}}\nolimits}
\newcommand{\Bk}{\mathop{\mathrm{Bk}}\nolimits}
\newcommand{\HOM}{\mathop{\mathscr{H}{\kern -3pt {om}}}\nolimits}
\newcommand{\calA}{\mathscr{A}}
\newcommand{\calC}{\mathscr{C}}
\newcommand{\calH}{\mathscr{H}}
\newcommand{\calL}{\mathscr{L}}
\newcommand{\calM}{\mathscr{M}}
\newcommand{\CM}{\mathcal{M}}
\newcommand{\calN}{\mathscr{N}}
\newcommand{\calX}{\mathscr{X}}
\newcommand{\calK}{\mathscr{K}}
\newcommand{\calD}{\mathscr{D}}
\newcommand{\calY}{\mathscr{Y}}
\newcommand{\calF}{\mathscr{F}}
\newcommand{\CN}{\mathcal{N}}
\newcommand{\DD}{\mathcal{D}}
\newcommand{\CCC}{\mathcal{C}}

\newcommand{\chari}{\mathop{\mathrm{char}}\nolimits}
\newcommand{\ch}{\mathop{\mathrm{ch}}\nolimits}
\newcommand{\CH}{\mathop{\mathrm{CH}}\nolimits}
\newcommand{\supp}{\mathop{\mathrm{supp}}\nolimits}
\newcommand{\codim}{\mathop{\mathrm{codim}}\nolimits}
\newcommand{\td}{\mathop{\mathrm{td}}\nolimits}
\newcommand{\Span}{\mathop{\mathrm{Span}}\nolimits}
\newcommand{\Gal}{\mathop{\mathrm{Gal}}\nolimits}
\newcommand{\sym}{\mathop{\mathrm{Sym}}\nolimits}
\newcommand{\DIV}{\mathop{\mathrm{div}}\nolimits}
\newcommand{\Gr}{\mathop{\mathrm{Gr}}\nolimits}
\newcommand{\cl}{\mathop{\mathrm{cl}}\nolimits}
\newcommand{\wcl}{\mathop{\widetilde{\mathrm{cl}}}\nolimits}
\newcommand{\piet}{{\pi_1^{\rm \acute{e}t}}}
\newcommand{\Het}[1]{{H_{\rm \acute{e}t}^{{#1}}}}
\newcommand{\Hfl}[1]{{H_{\rm fl}^{{#1}}}}
\newcommand{\Hcris}[1]{{H_{\rm cris}^{{#1}}}}
\newcommand{\HdR}[1]{{H_{\rm dR}^{{#1}}}}
\newcommand{\hdR}[1]{{h_{\rm dR}^{{#1}}}}
\newcommand{\loc}{{\rm loc}}
\newcommand{\et}{{\rm \acute{e}t}}
\newcommand{\defin}[1]{{\bf #1}}

\newcommand{\ol}[1]{{\overline{#1}}}

\ifthenelse{\equal{1}{1}}{
\ifthenelse{\equal{2}{2}}{
\newcommand{\blue}[1]{{\color{blue}#1}}
\newcommand{\green}[1]{{\color{green}#1}}
\newcommand{\red}[1]{{\color{red}#1}}
\newcommand{\cyan}[1]{{\color{cyan}#1}}
\newcommand{\magenta}[1]{{\color{magenta}#1}}
\newcommand{\yellow}[1]{{\color{yellow}#1}} 
}{
\newcommand{\blue}[1]{#1}
\newcommand{\green}[1]{#1}
\newcommand{\red}[1]{#1}
\newcommand{\cyan}[1]{#1}
\newcommand{\magenta}[1]{#1}
\newcommand{\yellow}[1]{#1} 
}
}{
\newcommand{\blue}[1]{}
\newcommand{\green}[1]{}
\newcommand{\red}[1]{}
\newcommand{\cyan}[1]{}
\newcommand{\magenta}[1]{}
\newcommand{\yellow}[1]{} 
}

\renewcommand{\HH}{{\rm{H}}}

\title[Real Regulators for Products]{Real Regulators for Products of Elliptic Curves}

\author{Xi Chen}
\address{632 Central Academic Building\\
University of Alberta\\
Edmonton, Alberta T6G 2G1, CANADA}
\email{xichen@math.ualberta.ca}

\author{James D. Lewis}
\address{632 Central Academic Building\\ University of Alberta\\ Edmonton, Alberta T6G 2G1, CANADA}
\email{lewisjd@ualberta.ca}

\thanks{Both authors partially supported
by grants from the Natural Sciences and Engineering
Research Council of Canada}

\date{\today}

\begin{abstract}
Assuming the K\"unneth decomposition of the Chow groups of products of very general Kummer surfaces, we prove that the Hodge-$\calD$-conjecture fails for the real regulator $r_{k,1}$ on a product of $n$ very general elliptic curves for $2n\ge 3k-1\ge 8$.
\end{abstract}

\maketitle

\section{Introduction}

Let $X$ be a projective variety and let $(k,m)$ be a pair of integers. The higher Chow groups $\CH^{k}(X,m)$ were introduced by S. Bloch \cite{Bloch1986AlgebraicCA}.
For the purpose of this paper, let us give a quick definition of $\CH^{k}(X,1)$ as follows
$$
\begin{aligned}
Z^k(X,1) &= \Big\{
\sum_{j}(f_{j},Z_{j})\ : \text{cd}_{X}Z_{j} = k-1,\ f_{j}\in \CC(Z_{j})^{\times}\Big\}\\
\CH^{k}(X,1) &= \frac{\big\{ \xi = \sum (f_j, Z_j)\in Z^k(X,1): \DIV(\xi) = \sum \DIV(f_{j}) = 0\big\} }{\text{Image(Tame symbol)}}.
\end{aligned}
$$
where $Z_j$ are irreducible subvarieties of $X$ of codimension $k-1$, $\CC(Z_j)$ is the space of rational functions on $Z_j$ and $\DIV(f_j)$ is the divisor on $Z_j$ defined by $f_j$. We will not explain the Tame symbol since it is not needed in this paper.

If $X$ is smooth, then similar to the cycle maps $\CH^k(X)\to H^{2k}(X)$ on Chow groups, there are maps, called {\em regulators}, from the higher Chow groups of $X$ to its Deligne cohomologies (see, for example, \cite{kerr_lewis_muller-stach_2006, K-L}). Again, for our purpose, we just need the {\em real regulator} map 
$$
\begin{tikzcd}
\CH_\RR^k(X, 1) \ar[equal]{d} \ar{r}{r_{k,1}} & H^{k-1,k-1}(X, \RR) \ar[equal]{d} \\
\CH^k(X,1)\otimes \RR & (H^{n-k+1, n-k+1}(X, \RR))^\vee
\end{tikzcd}
$$
defined on $\CH^k(X,1)$ for a smooth projective variety $X$ of dimension $n$, which is explicitly given by
$$
r_{k,1}(\xi)(\omega) = \sum_j \int_{Z_j} \log|f_j| \omega
$$
for $\xi\in \CH^k(X,1)$ represented by $\xi = \sum (f_j, Z_i)\in Z^k(X,1)$ satisfying $\DIV(\xi) = 0$.

The Hodge-$\calD$-conjecture states that this map is surjective. It is expected to be true for varieties over $\overline{\QQ}$. For surfaces over $\CC$, it is known to be true for rational surfaces and general Abelian and K3 surfaces \cite{HodgeDK3}. It fails for very general surfaces in $\PP^3$ of degree $\ge 5$ \cite{SMS1995}.

Let us consider the real regulator for a product of elliptic curves.

\begin{Conjecture}\label{HDPECONJPE}
For $n$ very general complex elliptic curves $E_1,E_2,...,E_n$ and $X = E_1\times E_2\times ...\times E_n$, the real regulator map
$r_{k,1}$ on $X$ is surjective for $k = 2$ and ``trivial'' (explained below) for all $2n\ge 3k-1\ge 8$.
The triviality of $r_{k,1}$ is measured by whether its image is orthogonal to one of the subspaces
\begin{equation}\label{HDPEE803}
T_m(H^{2n-2k+2}(X)) = \sum_{\substack{|I|=m\\ I\subset \{1,2,...,n\}}} \bigotimes_{i\in I} H^1(E_i)
\otimes 
H^{2n - 2k + 2 - m}(\prod_{j\not\in I} E_j)
\end{equation}
of $H^{2n-2k+2}(X)$ for some $1\le m\le n$. We expect that
\begin{equation}\label{HDPEE800}
r_{k,1}(\CH_\RR^{k}(X, 1))
\subset T_{2r+2}(H^{2n - 2k + 2}(X))^\perp
\end{equation}
for all $3\le k\le 2r + 1$. For example, when $(k,r,n) = (3,1,4)$, we expect that
$$
r_{3,1}(\CH_\RR^{3}(X, 1))
\subset T_4(H^{4}(X))^\perp = (H^1(E_1)\otimes H^1(E_2)\otimes H^1(E_3)\otimes H^1(E_4))^\perp.
$$
\end{Conjecture}

It is easy to show that $r_{k,1}$ is surjective if and only if
\begin{equation}\label{HDPEE000}
\begin{aligned}
& \bigotimes_{i=1}^{2l} H^1(E_{a_i}) \cap H^{l,l}(E_{a_1}\times E_{a_2}\times ... \times E_{a_{2l}}, \RR)
\\
&\subset r_{l+1,1}(\CH_\RR^{l+1}(E_{a_1}\times E_{a_2}\times ... \times E_{a_{2l}}, 1))
\end{aligned}
\end{equation}
for all $1\le l \le k-1$ and $1\le a_1 < a_2 < ... < a_{2l} \le n$. This holds for $k=2$ \cite{HodgeDK3}. So the question here is whether $r_{k,1}$ is trivial for $3\le k \le n-1$.

In this paper, we will reduce the question regarding regulators on products of elliptic curves to those on products of Kummer surfaces. Here a {\em Kummer surface} is the minimal resolution of $E_1\times E_2/\pm 1$ for a product of elliptic curves $E_1$ and $E_2$. Roughly, if we assume a K\"unneth decomposition of some Chow group of products of very general Kummer surfaces, then we have the triviality of $r_{k,1}$.

\begin{Theorem}\label{HDPETHMMAIN}
Let $k, r$ and $m$ be integers satisfying $3\le k \le 2r + 1 \le 2m-1$.
Suppose that the natural map
\begin{equation}\label{HDPEE801}
\begin{tikzcd}
\displaystyle{\bigoplus_{\sum d_i = k-1} \bigotimes_{j=1}^r \CH_\QQ^{d_j}(Y_j) \otimes \CH_\QQ^{d_{r+1}} (\prod_{j=r+1}^m Y_j)}
\ar[->>]{r}& \displaystyle{\CH_\QQ^{k-1}(\prod_{j=1}^m Y_j)}
\end{tikzcd}
\end{equation}
is surjective
for very general Kummer surfaces $Y_1, ..., Y_r, Y_{r+2}, ..., Y_m$ and all Kummer surfaces $Y_{r+1}$ with $\rank_\ZZ \Pic(Y_{r+1}) \le 19$. Then \eqref{HDPEE800} holds on a product $X$ of $n\le 2m$ very general elliptic curves.
\end{Theorem}

In the above theorem, for example, if $n=2m=4$, $r=1$ and $k=3$, then
$$
\sum_{i=1}^2 d_i = 2.
$$
And \eqref{HDPEE801} becomes
$$
\begin{tikzcd}
\displaystyle{\sum_{d_1+d_2=2} \CH_\QQ^{d_1}(Y_1) \otimes \CH_\QQ^{d_2}(Y_2)}
\ar[->>]{r}& \displaystyle{\CH_\QQ^{2}(Y_1\times Y_2)},
\end{tikzcd}
$$
which is equivalent to the K\"unneth decomposition of $\CH_\QQ^2(Y_1\times Y_2)$.

Keep in mind the difference between $Y_{r+1}$ and the rest of $Y_j$.

Note that if \eqref{HDPEE800} holds for $n = n_0$, then we see that it holds for all $k+1\le n\le n_0$ by projecting $E_1\times E_2\times ... \times E_{n_0}$ to $E_1\times E_2 \times ... \times E_n$.
So we just have to prove the above theorem for $n=2m$.

We will prove Theorem \ref{HDPETHMMAIN} in sections \ref{HDPESECCHCC} and \ref{HDPESECPKS}. In section \ref{HDPESECBBCAJ}, we will show that the K\"unneth decomposition
of $\CH_\QQ^2(Y_1\times Y_2)$ is a consequence of the Bloch-Beilinson conjecture on Abel-Jacobi maps. Hence \eqref{HDPEE800} holds for $(k,r, n) = (3,1,4)$ if we assume the Bloch-Beilinson conjecture. Consequently, either the Hodge-$\calD$-conjecture fails for $r_{3,1}$ on a product of four very general elliptic curves or the Bloch-Beilinson conjecture fails.

We work exclusively over $\CC$ unless otherwise stated.

\medskip
We are grateful to the referee for doing a splendid job.

\section{Completion of Higher Chow Cycles}\label{HDPESECCHCC}

Roughly speaking, we will follow the same argument in \cite{ChenLewis4}. That is, we will construct a family of products of Kummer surfaces, extend a higher Chow cycle to the whole family and use a standard monodromy argument to show that it has trivial regulator on a general fiber.
First we need a generalization of \cite[Theorem 0.1]{ChenLewis4}, which allows us to extend a higher Chow cycle over the family after some modification.

\begin{Theorem}\label{HDPETHM000}
Let $f: W \to \Gamma$ be a dominant morphism with connected fibers from a smooth projective
variety $W$ to a smooth projective curve $\Gamma$, let $Y_1,Y_2,...,Y_r$ be smooth projective surfaces with $H^1(Y_j) = 0$ and let $k \le 2r+1$ be a positive integer such that the natural map
\begin{equation}\label{HDPEE903}
\begin{tikzcd}
\displaystyle{\bigoplus_{\sum d_i = k-1} \bigotimes_{j=1}^r \CH_\QQ^{d_j}(Y_j) \otimes \CH_\QQ^{d_{r+1}}(V)}
\ar[->>]{r} & \displaystyle{\CH_\QQ^{k-1}(\prod_{j=1}^r Y_j\times V)}
\end{tikzcd}
\end{equation}
is surjective for all irreducible components $V$ of $W_t$ and all $t\in \Gamma$, where $W_t$ is the fiber of $W/\Gamma$ over $t$.
Let $\xi\in \CH_\QQ^k(\prod Y_j \times W_U, 1)$ be a higher Chow cycle defined on $\prod Y_j \times W_U = \prod Y_j \times (W\times_\Gamma U)$ for an open set $U\subset \Gamma$. Then there exist $\eta\in \CH_\QQ^{k}(\prod Y_j \times W, 1)$ and pre-higher Chow cycles $\alpha_0,\alpha_1, ...,\alpha_r$ on $\prod Y_j \times W$ such that
\begin{equation}\label{HDPEE909}
\alpha_0 \in f^* Z_\QQ^k(\prod Y_j \times \Gamma, 1),
\end{equation}
\begin{equation}\label{HDPEE904}
\alpha_i \in Z_\QQ^0(Y_i)\otimes Z_\QQ^{k}(\prod_{j\ne i} Y_j\times W, 1)
\oplus Z_\QQ^1(Y_i)\otimes Z_\QQ^{k-1}(\prod_{j\ne i} Y_j\times W, 1)
\end{equation}
for $i=1,2,...,r$, and
$$
\eta + \sum_{i=0}^{r} \alpha_i = \xi
$$
on $\prod Y_j\times W_U$, where $Z^m(X)$ is the free abelian group of Chow cycles of codimension $m$ on $X$ and
$Z^m(X, 1)$ is the free abelian group of pre-higher Chow cycles defined at the beginning.
\end{Theorem}

\begin{proof}
We can extend $\xi$ to a pre-higher Chow cycle $\overline{\xi}$ on $\prod Y_j\times W$ with $\DIV(\overline{\xi})$ supported on $\prod Y_j\times W_B$ for $B = \Gamma \backslash U$.
By the surjection \eqref{HDPEE903}, 
we may choose the completion $\overline{\xi}$ of $\xi$, after some modification by a cycle in $Z_\QQ^{k-1}(\prod Y_j\times W_B,1)$, such that
$$
\DIV(\overline{\xi}) = \sum_{i=1}^r R_i + R_0
$$
where 
$$
\begin{aligned}
& R_i \in \bigotimes_{a=1}^{i-1} Z_\QQ^2(Y_a) \otimes \Big(Z_\QQ^0(Y_i)
\otimes Z_\QQ^{k-2i+1}(\prod_{j=i+1}^r Y_j\times W_B)
\\
&\hspace{96pt}
\oplus Z_\QQ^1(Y_i)
\otimes Z_\QQ^{k-2i}(\prod_{j=i+1}^r Y_j\times W_B)
\Big)
\end{aligned}
$$
for $i=1,2,...,r$ and
$$
R_0 \in \bigotimes_{a=1}^{r} Z_\QQ^2(Y_a) \otimes Z_\QQ^{k-2r-1}(W_B).
$$
Note that 
$$
\sum_{i=0}^r R_i \sim_\text{rat} 0
$$
in $\CH_\QQ^{k}(\prod Y_j\times W)$. 

Let us prove by induction that $R_i = \DIV(\alpha_i)$ 
for $i = 1,2,...,r$ and some $\alpha_i$ in the space \eqref{HDPEE904}.

Starting with $R_1$, we write
$$
\begin{aligned}
R_1 = R_{1,0} + R_{1,1}\hspace{24pt} &\text{for } R_{1,0} \in Z_\QQ^0(Y_1)
\otimes Z_\QQ^{k-1}(\prod_{j=2}^r Y_j\times W_B)\\
&\hspace{15pt} R_{1,1} \in Z_\QQ^1(Y_1)
\otimes Z_\QQ^{k-2}(\prod_{j=2}^r Y_j\times W_B)
\end{aligned}
$$
Clearly, $R_{1,0}$ can be written as
$$
R_{1,0} = Y_1 \otimes S \hspace{24pt} \text{for some } S\in Z_\QQ^{k-1}(\prod_{j=2}^r Y_j\times W_B)
$$
By intersecting $\sum R_i$ with $p\times Y_2\times ... \times Y_r \times W$ for a point $p\in Y_1$, we see that $S \sim_\text{rat} 0$ in $\CH_\QQ^{k}(\prod_{j=2}^r Y_j\times W)$. Therefore,
$$
R_{1,0} = \DIV(\alpha_{1,0}) \hspace{24pt} \text{for some } \alpha_{1,0}\in Z_\QQ^0(Y_1) \otimes Z_\QQ^{k}(\prod_{j=2}^r Y_j\times W, 1)
$$ 
Hence $R_{1,0} \sim_\text{rat} 0$ and
$$
R_{1,1} + \sum_{i=2}^r R_i + R_0 \sim_\text{rat} 0
$$
in $\CH_\QQ^{k}(\prod Y_j\times W)$. 

We can write
$$
R_{1,1} = \sum L_a \otimes S_a \hspace{12pt}\text{for some } L_a\in Z_\QQ^1(Y_1) \text{ and } S_a\in Z_\QQ^{k-2}(\prod_{j=2}^r Y_j\times W_B)
$$
We may assume that $L_a$ are linearly independent in $\CH_\QQ^1(Y_1)$, after further modifying $\overline{\xi}$ by some cycle supported on $\prod Y_j\times W_B$.

Since $H^1(Y_1) = 0$, the intersection pairing
$$
\begin{tikzcd}
\CH_\QQ^1(Y_1) \otimes \CH_\QQ^1(Y_1) \ar{r} & H^2(Y_1,\QQ)\cong \QQ
\end{tikzcd}
$$
is nondegenerate. Therefore, by intersecting $R_{1,1} + R_2 + ... + R_r + R_0$ with cycles in
$$
Z_\QQ^1(Y_1)\otimes Z_\QQ^0(\prod_{j=2}^r Y_j\times W),
$$
we obtain $S_a \sim_\text{rat} 0$ in $\CH_\QQ^{k-1}(\prod_{j=2}^r Y_j\times W)$ for all $a$. Consequently,
$$
R_{1,1} = \DIV(\alpha_{1,1}) \hspace{24pt} \text{for some } \alpha_{1,1}\in Z_\QQ^1(Y_1) \otimes Z_\QQ^{k-1}(\prod_{j=2}^r Y_j\times W, 1)
$$
Then
$$
R_1 = \DIV(\alpha_{1,0} + \alpha_{1,1}) = \DIV(\alpha_1)
$$
and hence
$$
\sum_{i=2}^r R_i + R_0 \sim_\text{rat} 0
$$
in $\CH_\QQ^{k}(\prod Y_j\times W)$. 

In this way, we can inductively show that $R_i = \DIV(\alpha_i)$ for $i=1,2,...,r$ by intersecting $R_i + ... + R_r + R_0$ with cycles in
$$
Z_\QQ^2(Y_i) \otimes Z_\QQ^0(\prod_{j\ne i} Y_j\times W)
\hspace{12pt}\text{and}\hspace{12pt}
Z_\QQ^1(Y_i) \otimes Z_\QQ^0(\prod_{j\ne i} Y_j\times W).
$$
It follows that
$$
R_0 \sim_\text{rat} 0
$$
in $\CH_\QQ^{k}(\prod Y_j\times W)$. It remains to find $\alpha_0$ in the space \eqref{HDPEE909} such that $R_0 = \DIV(\alpha_0)$.

If $k < 2r+1$, then $R_0 = 0$ and there is nothing to prove. Suppose that
$k = 2r+1$. In this case,
$$
R_0 \in Z_\QQ^{2r}(Y) \otimes Z_\QQ^0(W_B) \hspace{24pt} \text{for } Y = \prod_{j=1}^r Y_j.
$$
Let us write
$$
R_0 = \sum L_a \otimes S_a
$$
where $S_a$ are irreducible components of $W_B$ and $L_a\in Z_\QQ^{2r}(Y)$. Let $\mu_a$ be the multiplicity of $S_a$ in $W_B$. We claim that for every pair $S_a$ and $S_b$ with $f(S_a) = f(S_b)$, i.e., for any two components $S_a$ and $S_b$ of $W_p$ and all $p\in B$,
$$
\mu_b L_a \sim_\text{rat} \mu_a L_b
$$
over $\QQ$ on $Y$.

Since $W$ is smooth, the components of $W_B$ are Cartier divisors of $W$. We take a sufficiently ample divisor $A$ on $W$ and cut $W$ by $n-2$ general members $A_1,A_2,...,A_{n-2}\in |A|$ for $n=\dim W$. The resulting $D = A_1\cap A_2\cap ...\cap A_{n-2}$ is a smooth projective surface and a flat family of curves over $\Gamma$. The basic intersection theory on surfaces tells us that for every $p\in B$, the intersection matrix of any $m-1$ irreducible components of $D_p$ is negative definite, where $m$ is the number of irreducible components of $W_p$. Therefore, for any two components $S_a$ and $S_b$ of $W_p$, there exists $\Lambda\in Z^1(W)$, supported on $W_p$, such that 
$$
\begin{aligned}
\Lambda.D.S &\equiv 0 \text{ for all components } S\ne S_a, S_b \text{ of } W_p,\\
\Lambda.D.S_a &\not\equiv 0, \text{ and } \Lambda.D.(\mu_a S_a + \mu_b S_b) = 0
\end{aligned}
$$
where ``$\equiv$'' is numerical equivalence. And since $\Lambda$ is supported on $W_p$, we actually have $\Lambda.D.S_i \equiv 0$ for all $i\ne a, b$. For simplicity, by choosing the cycle $\Lambda\in Z_\QQ^1(W)$ over $\QQ$, we may assume that $\Lambda.D.S_a \equiv \mu_b$. In summary, by letting $C = \Lambda.D$, we conclude that
for every $p\in B$ and any two components $S_a$ and $S_b$ of $W_p$, we can find a $1$-cycle $C\in Z_\QQ^{n-1}(W)$ such that 
$$
C.S_i \equiv 0 \text{ for } i \ne a, b,\ C.S_a \equiv \mu_b, \text{ and }
C.S_b \equiv -\mu_a.
$$
Then
$$
\begin{aligned}
f_*((Y\otimes C) . R_0) &=  (C.S_a) L_a \otimes p + (C.S_b) L_b \otimes p\\
&= ((C.S_a) L_a + (C.S_b) L_b) \otimes p
\end{aligned}
$$
for $f: Y\times W\to Y\times \Gamma$. Thus
$$
(C.S_a) L_a + (C.S_b) L_b \sim_\text{rat} 0 \Rightarrow \mu_b L_a - \mu_a L_b \sim_\text{rat} 0.
$$
Therefore, $\mu_b L_a \sim_\text{rat} \mu_a L_b$ for all pairs of components $S_a$ and $S_b$ of $W_p$. This implies that after replacing $\overline{\xi}$ by $\overline{\xi} + \beta$ for some $\beta\in Z_\QQ^0(W_B)\otimes Z_\QQ^{k-1}(Y, 1)$, we may write $R_0$ as
$$
R_0 = \sum_{p\in B} M_p\otimes W_p
$$
where $M_p = (1/\mu_a) L_a$ for a component $S_a$ of $W_p$. Namely, $R_0 = f^* G$ for some $G\in Z_\QQ^{k-1}(Y)\otimes Z_\QQ^1(\Gamma)$. Since $R_0\sim_\text{rat} 0$ on $Y\times W$, $G\sim_\text{rat} 0$ on $Y\times \Gamma$. So there exists $\alpha_0 \in f^* Z_\QQ^k(Y \times \Gamma, 1)$ such that $R_0 = \DIV(\alpha_0)$.

In conclusion,
$$
\eta = \overline{\xi} - \sum_{i=0}^{r} \alpha_i
$$
is a higher Chow cycle in $\CH_\QQ^k(Y\times W, 1)$ with the required property.
\end{proof}

\section{Products of Kummer Surfaces}\label{HDPESECPKS}

We will reduce the triviality of $r_{k,1}$ on products of elliptic curves to that on products of Kummer surfaces.

For a product $E_1\times E_2$ of two elliptic curves, we fix two involutions $\sigma_1$ and $\sigma_2$ on $E_i$ and let $E_1\times E_2 / \sigma_1\times\sigma_2$ be the quotient of $E_1\times E_2$ by the action $\sigma_1\times \sigma_2$. Usually, we simply write it as $E_1\times E_2 / \pm 1$.
Note that the action of $\sigma_1\times \sigma_2$ is invariant on $H^2(E_1\times E_2)$, i.e.,
\begin{equation}\label{HDPEE900}
(\sigma_1\times \sigma_2) \omega = \omega \hspace{12pt} \text{for all } \omega\in H^2(E_1\times E_2)
\end{equation}

The resulting surface $E_1\times E_2 / \pm 1$ has $16$ ordinary double points, corresponding to $16$ fixed points of $\sigma_1\times \sigma_2$. Blowing up at the $16$ double points, we obtain a Kummer K3 surface $Y$. Indeed, we have a diagram
\begin{equation}\label{HDPEE906}
\begin{tikzcd}
Z \ar{r}{f} \ar{d}[left]{g} & Y\ar{d}\\
X \ar{r} & X/\pm 1
\end{tikzcd}
\end{equation}
where $X = E_1\times E_2$, $Z$ is the blowup of $X$ at the $16$ fixed points of $\sigma_1\times \sigma_2$ and $f$ is a finite morphism of degree $2$ ramified at the $16$ exceptional divisors of $Z\to X$. The action $\sigma_1\times \sigma_2$ on $X$ extends to the Galois action $\sigma$ on $Z$ associated to $f$. Clearly, $\sigma$ preserves the $16$ exceptional divisors of $g$. Combining this with \eqref{HDPEE900}, we see that
$$\sigma (\omega) = \omega \hspace{12pt} \text{for all } \omega\in H^2(Z).$$
Thus, we have
\begin{equation}\label{HDPEE901}
f^* f_* \omega = \omega + \sigma(\omega) = 2\omega \hspace{12pt} \text{for all } \omega\in H^2(Z).
\end{equation}
Combined with the projection formula $f_*f^* \omega = 2\omega$, \eqref{HDPEE901} implies that $f^*$ and $f_*$ are isomorphims between $H^2(Y)$ and $H^2(Z)$ satisfying (with $\deg f= 2$)
\begin{equation}\label{HDPEE902}
\begin{tikzcd}
H^2(Y)\ar{r}{f^*}[below]{\sim} \ar[bend right=25]{rr}[below]{(\deg f) I} & H^2(Z) \ar[bend left=25]{rr}{(\deg f) I}\ar{r}{f_*}[below]{\sim} & H^2(Y)
\ar{r}{f^*}[below]{\sim} & H^2(Z)
\end{tikzcd}
\end{equation}
It follows that $(1/\deg f) f_* = (f^*)^{-1}$ preserves the intersection pairing and hence
\begin{equation}\label{HDPEE907}
\langle f_* \alpha, f_* \beta\rangle = (\deg f) f_* \langle \alpha, \beta\rangle
\end{equation}
for all $\alpha, \beta\in H^2(Z)$.

Furthermore, $f^*$ and $f_*$ induce isomorphisms between the $\QQ$-Hodge structures on $H^2(Y,\QQ)$ and $H^2(Z,\QQ)$. Thus they induce isomorphims between the algebraic/transcendental parts of $H^2(Y)$ and $H^2(Z)$.

We define
\begin{equation}\label{HDPEE905}
H_{\mathrm{tr}}^2(Y) = f_* g^* (H^1(E_1)\otimes H^1(E_2)).
\end{equation}
Strictly speaking, this is not exactly the transcendental part of $H^2(Y)$. It is the subspace orthogonal to the $18$ algebraic classes of $H^2(Y)$ corresponding to the two fibers of $E_1\times E_2$ over $E_i$ and $16$ exceptional divisors of $g$. 
For very general $E_1$ and $E_2$, this is the transcendental part of $H^2(Y)$. For arbitrary $E_1$ and $E_2$, it contains the transcendental part of $H^2(Y)$ as a subspace.

Based on the above observations, we have

\begin{Proposition}\label{HDPEPROP900}
Let $E_1,E_2,...,E_{2m}$ be $n = 2m$ elliptic curves and let $Y_{ij}$ be the Kummer surface birational to $E_i\times E_j/\pm 1$. Then \eqref{HDPEE800} holds if the real regulator $r_{k,1}$ on
$$
Y = Y_{a_1 a_2} \times Y_{a_3 a_4} \times ... \times Y_{a_{2m-1} a_{2m}} 
$$
satisfies
$$
r_{k,1}(\CH_\RR^k(Y,1)) \subset \Big(\bigotimes_{i=1}^{r+1} H_{\mathrm{tr}}^2(Y_{a_{2i-1}a_{2i}}) \otimes H^{4m-2k - 2r}(\prod_{i=r+2}^m Y_{a_{2i-1}a_{2i}})\Big)^\perp
$$
for all $\{a_1, a_2, ..., a_{2m}\} = \{1,2,...,2m\}$, where $H_{\mathrm{tr}}^2(Y_{ij})$ is the subspace of $H^2(Y_{ij})$ defined by \eqref{HDPEE905}.
\end{Proposition}

\begin{proof}
Clearly, $T_{2r+2}(H^{4m-2k+2}(X))$ is spanned by the forms
$$
\underbrace{\omega_1 \otimes \omega_2 \otimes ... \otimes \omega_{2r+2}}_\omega \otimes \eta
$$
for some $\omega_i\in H^1(E_{a_i})$ and
$$
\eta \in H^{4m - 2k-2r}(\prod_{i\ne a_1,...,a_{2r+2}} E_i).
$$
It suffices to prove that
\begin{equation}\label{HDPEE908}
r_{k,1}(\CH_\RR^k(X,1)) \subset (\omega\otimes \eta)^\perp
\end{equation}
For simplicity, we may assume that $(a_1,a_2,...,a_{2r+2}) = (1,2,...,2r+2)$.

Let $Z_{ij}$ be the blowup of $E_i\times E_j$ at the $16$ fixed points and let
$$
Y = Y_{12}\times Y_{34}\times ... \times Y_{2m-1,2m} 
$$
and
$$
Z = Z_{12}\times Z_{34}\times ... \times Z_{2m-1,2m}.
$$
We have the commutative diagram \eqref{HDPEE906}. 

By \eqref{HDPEE902} and \eqref{HDPEE907}, 
$$
\begin{aligned}
&\quad\langle r_{k,1}(\xi), \omega \otimes \eta\rangle = g^* \langle r_{k,1}(\xi), \omega \otimes \eta\rangle = \langle g^* (r_{k,1}(\xi)), g^*(\omega \otimes \eta)\rangle\\
&= f_* \langle r_{k,1}(g^* \xi), g^*(\omega \otimes \eta)\rangle = \frac{1}{\deg f} \langle f_*(r_{k,1}(g^* \xi)), f_* g^*(\omega \otimes \eta)\rangle\\
&= \frac{1}{\deg f} \langle r_{k,1}(f_* g^* \xi), f_* g^*(\omega \otimes \eta)\rangle = 0
\end{aligned}
$$
for all $\xi\in \CH^k(X, 1)$ and \eqref{HDPEE908} follows.
\end{proof}

By the above proposition, to prove the triviality of $r_{k,1}$ on a product of $2m$ very general elliptic curves in our main Theorem \ref{HDPETHMMAIN}, it suffices to prove
\begin{equation}\label{HDPEE804}
r_{k,1}(\CH_\RR^k(\prod_{i=1}^m Y_i, 1)) \subset \Big(\bigotimes_{i=1}^{r+1} H_{\mathrm{tr}}^2(Y_i) 
\otimes H^{4m-2k-2r}(\prod_{i=r+2}^m Y_{i})\Big)^\perp
\end{equation}
for a product of $m$ very general Kummer surfaces $Y_1,Y_2,...,Y_m$.

Now let us try to use the argument in \cite{ChenLewis4} to prove \eqref{HDPEE804} and thus Theorem \ref{HDPETHMMAIN}. As in \cite{ChenLewis4}, we first construct a one-parameter family of Kummer surfaces with ``nice'' singularities.

We start with the construction of two flat projective families $S/B$ and $T/B$ of curves over a smooth projective curve $B$ satisfying
\begin{itemize}
\item $S$ and $T$ are smooth,
\item there is a nonempty finite set $\Sigma\subset B$ such that $S_b$ and $T_b$ are rational curves with a node for $b\in \Sigma$ and they are smooth elliptic curves for $b\not\in \Sigma$,
\item $S_b\times T_b$ is a general product of two elliptic curves for $b\in B$ general,
\begin{equation}\label{HDPEE802}
h^{1,1}(S_b\times T_b, \QQ) := \dim H^{1,1}(S_b\times T_b, \QQ) \le 3 \hspace{24pt} \text{for all } b\in B\backslash \Sigma,
\end{equation}
\item and both $S/B$ and $T/B$ have sections.
\end{itemize}

By \eqref{HDPEE802}, the one-parameter family of Kummer surfaces constructed from $S\times_B T$ is generically of Picard rank $18$ and of Picard rank $19$ (but never $20$) at finitely many points of $B\backslash \Sigma$.

It is not hard to construct such $S/B$ and $T/B$ individually. The difficulty is that we have to make sure that $S/B$ and $T/B$ are singular over the same points $b\in B$. Here is one construction.

We let $G\subset \PP^2\times \PP^1$ be a general pencil of cubic curves. It is well known that $G/\PP^1$ has exactly $12$ nodal fibers over $p_1,p_2,...,p_{12}\in \PP^1$. We choose two different morphisms $g_i: \PP^1\to \PP^1$ of degree $12$ that map all $p_1, p_2,...,p_{12}$ to the same point $q\in \PP^1$, i.e.,
$$
g_i^*(q) = p_1 + p_2 + ... + p_{12} \hspace{24pt} \text{for } i=1,2.
$$
Let $B$ be the normalization of the fiber product of $g_1: \PP^1\to \PP^1$ and $g_2: \PP^1\to \PP^1$ with the diagram
$$
\begin{tikzcd}
& B \ar{dl}[above left]{\pi_1} \ar{dr}{\pi_2} \ar{dd}{\pi}\\
\PP^1 \ar{dr}[below left]{g_1} & & \PP^1 \ar{dl}{g_2}\\
& \PP^1
\end{tikzcd}
$$
Then
$$
\Sigma = \pi^{-1}(q) = \pi_1^{-1}\{p_1,p_2,...,p_{12}\} = \pi_2^{-1}\{p_1,p_2,...,p_{12}\}.
$$
Indeed, for general choices of $g_1$ and $g_2$, $B$ is irreducible and $\pi: B\to \PP^1$ has degree $144$.

Let $S = G\times_{\PP^1} B$ be the fiber product of $G\to \PP^1$ and $\pi_1: B\to \PP^1$ and let $T = G\times_{\PP^1} B$ be the fiber product of $G\to \PP^1$ and $\pi_2: B\to \PP^1$. It is not hard to see that $S/B$ and $T/B$ have the required properties for very general choices of $g_1$ and $g_2$: \eqref{HDPEE802} holds since there are only countably many products of elliptic curves $E\times F$ with $h^{1,1}(E\times F, \QQ) = 4$; it is easy to choose $g_i$ such that $S_b\times T_b$ is not one of them for all $b\in B\backslash \Sigma$; the pencil $G/\PP^1$ has infinitely many sections so the same holds for both $S/B$ and $T/B$. This shows the existence of such $S/B$ and $T/B$.

Since $S/B$ has a section, we have an involution $\sigma_S: S/B\dashrightarrow S/B$ defined on smooth fibers of $S/B$. This involution extends to singular fibers of $S/B$ as well:  for a nodal fiber $S_b$, it extends to an automorphism $S_b\to S_b$ fixing three points including the node. Indeed, this is the Galois action induced by a degree $2$ map $S_b\to \PP^1$.
So we have an automorphism $\sigma_S: S\to S$ preserving the base $B$ of order $2$. The fixed locus of $\sigma_S$ consists of a multisection of $S/B$ which meets each smooth fiber transversely at $4$ points and each singular fiber at $3$ points including the node.
Of course, the same holds for $T/B$ and we have an involution $\sigma_T: T/B\to T/B$.

Let $R = S\times_B T / \sigma$ for $\sigma = \sigma_S \times \sigma_T$. After resolving the singularities of $R$, we obtain a family of Kummer surfaces over $B$. 
Let $Y_1, ..., Y_r, Y_{r+2}, ..., Y_m$ be $m-1$ very general Kummer surfaces and let us try to prove \eqref{HDPEE804} where $Y_{r+1}$ is the Kummer surface birational to $R_t$ for $t\in B$ general. If \eqref{HDPEE804} fails, then there exist a finite base change $\phi: \Gamma\to B$, a desingularization $Z\rightarrow R\times_B \Gamma$ and a higher Chow cycle
$$
\xi\in \CH_\QQ^k(\prod_{i=1}^r Y_i \times Z_U \times \prod_{i=r+2}^m Y_i, 1)
$$
over a nonempty open set $U\subset \Gamma$ such that
\begin{itemize}
\item $Z_t$ is a Kummer surface birational to $R_{\phi(t)}$ for $t\not\in \phi^{-1}(\Sigma)$, and
\item for every $t\in U$,
$$
r_{k,1}(\xi_t) \not\in \big(\bigotimes_{i=1}^{r} H_{\mathrm{tr}}^2(Y_i) 
\otimes H_{\mathrm{tr}}^2(Z_t) \otimes H^{4m-2k-2r}(\prod_{i=r+2}^m Y_{i})\big)^\perp.
$$
\end{itemize}

We claim that we can choose $Z$, after a further finite base change $\Gamma'\to \Gamma$, such that
every irreducible component of $Z_t$ is a smooth rational surface for all $t\in \phi^{-1}(\Sigma)$. Namely, we claim

\begin{Proposition}\label{HDPEPROP000}
Let $R = S\times_B T / \sigma$ be constructed as above and let $\phi: \Gamma\to B$ be a finite morphism from a smooth projective curve $\Gamma$ to $B$ such that the ramification index of $\phi$ at each point of $\phi^{-1}(\Sigma)$ is even. Then there exists a desingularization $Z\to R\times_B \Gamma$ such that every irreducible component of $Z_t$ is a smooth rational surface for all $t\in \phi^{-1}(\Sigma)$.
\end{Proposition}

The hypothesis on the ramification index in the above proposition can be easily met by a further finite base change $\Gamma'\to \Gamma$.
Assuming Proposition \ref{HDPEPROP000}, let us finish the proof of \eqref{HDPEE804}.

For $t\not\in \phi^{-1}(\Sigma)$, by \eqref{HDPEE802}, we have $\rank_\ZZ \Pic(Z_t) \le 19$. Therefore,
by our hypothesis \eqref{HDPEE801}, the map
\begin{equation}\label{HDPEE809}
\begin{tikzcd}
\displaystyle{\bigotimes_{i=1}^r \CH_\QQ^\bullet(Y_i) \otimes \CH_\QQ^\bullet (Z_t\times\prod_{i=r+2}^m Y_i)}
\ar[->>]{r}& \displaystyle{\CH_\QQ^{k-1}(\prod_{i=1}^r Y_i\times Z_t\times \prod_{i=r+2}^m Y_i)}
\end{tikzcd}
\end{equation}
is surjective for all $t\not\in \phi^{-1}(\Sigma)$. 

For $t\in \phi^{-1}(\Sigma)$, every irreducible component $P\subset Z_t$ is a smooth rational surface by Proposition \ref{HDPEPROP000}. The Chow groups of $P\times X$ have K\"unneth decomposition
\begin{equation}\label{HDPEE805}
\CH^\bullet(P\times X) = \CH^\bullet(P) \otimes \CH^\bullet(X)
\end{equation}
for every smooth rational projective surface $P$ and every smooth projective variety $X$.

Let $Y_{r+1}$ be a general Kummer surface. Choosing a finite morphism $g: Y_{r+1}\to \PP^2$, we have the diagram
$$
\begin{tikzcd}
\displaystyle{\bigotimes_{i=1}^r \CH_\QQ^\bullet(Y_i) \otimes \CH_\QQ^\bullet (\prod_{i=r+1}^m Y_i)} \ar[->>]{r} \ar[->>]{d}{g_*} & \displaystyle{\CH_\QQ^{k-1}(\prod_{i=1}^m Y_i)} \ar[->>]{d}{g_*}\\
\displaystyle{\bigotimes_{i=1}^r \CH_\QQ^\bullet(Y_i) \otimes \CH_\QQ^\bullet (\PP^2\times \prod_{i=r+2}^m Y_i)} \ar{r} & \displaystyle{\CH_\QQ^{k-1}(\prod_{i=1}^r Y_i\times \PP^2\times \prod_{i=r+2}^m Y_i)}
\end{tikzcd}
$$
Clearly, we see from the above diagram that its bottom row is also surjective. Combining this with the K\"unneth decomposition \eqref{HDPEE805}, we have surjections
\begin{equation}\label{HDPEE810}
\begin{tikzcd}
\displaystyle{\bigotimes_{i=1}^r \CH_\QQ^\bullet(Y_i) \otimes \CH_\QQ^\bullet (\prod_{i=r+2}^m Y_i)} \ar[->>]{r} & \displaystyle{\CH_\QQ^{d}(\prod_{i=1}^r Y_i\times \prod_{i=r+2}^m Y_i)}
\end{tikzcd}
\end{equation}
for $d = k-1,k-2,k-3$. Then we obtain the surjection
\begin{equation}\label{HDPEE806}
\begin{tikzcd}
\displaystyle{\bigotimes_{i=1}^r \CH_\QQ^\bullet(Y_i) \otimes \CH_\QQ^\bullet (P\times \prod_{i=r+2}^m Y_i)} \ar[->>]{r} & \displaystyle{\CH_\QQ^{k-1}(\prod_{i=1}^r Y_i\times P\times \prod_{i=r+2}^m Y_i)}
\end{tikzcd}
\end{equation}
from \eqref{HDPEE805} and \eqref{HDPEE810} for every irreducible component $P\subset Z_t$ and all $t\in \phi^{-1}(\Sigma)$.

Combining \eqref{HDPEE809} and \eqref{HDPEE806}, we see that the map \eqref{HDPEE903} is surjective in Theorem \ref{HDPETHM000} for every irreducible component $V$ of $W_t$ and all $t\in \Gamma$ with
$$
W = Z\times \prod_{i=r+2}^m Y_{i}.
$$
So we can apply the theorem and obtain a higher Chow class
$$
\eta\in \CH_\QQ^k(\prod_{i=1}^r Y_i \times Z \times \prod_{i=r+2}^m Y_i, 1)
$$
and pre-higher Chow cycles $\alpha_0, \alpha_1, ..., \alpha_{r}$ as in the theorem such that
$$
\eta - \sum_{i=0}^{r} \alpha_i = \xi
$$
on $Y_1\times ... \times Y_r \times Z_U\times Y_{r+2}\times ... \times Y_m$. For $\alpha_0, \alpha_1, ..., \alpha_{r}$ given in Theorem \ref{HDPETHM000}, it follows from the explicit regulator formula applied to the precycles that
$$
r_{k,1}(\eta_t) - r_{k,1}(\xi_t)\in \Big(\bigotimes_{i=1}^{r} H_{\mathrm{tr}}^2(Y_i) 
\otimes H_{\mathrm{tr}}^2(Z_t) \otimes H^{4m-2k-2r}(\prod_{i=r+2}^m Y_{i})\Big)^\perp
$$
for all $t\in U$. Then a standard monodromy argument shows that $r_{k,1}(\xi_t)$ is trivial for $t\in U$ general (see, for example, \cite{ChenLewis4}). We will sketch this argument at the end of this section.

It remains to prove Proposition \ref{HDPEPROP000}. This is achieved by finding an explicit resolution of the singularities of $R\times_B \Gamma$.

\begin{proof}[Proof of Proposition \ref{HDPEPROP000}]
The problem is local at every point $\phi^{-1}(\Sigma)$. Let us replace $\Gamma$ by a disk centered at a point $0\in \phi^{-1}(\Sigma)$. So $S\times_B \Gamma$ and $T\times_B \Gamma$ have singularities of type $xy = t^{2m}$ at the nodes of $S_{\phi(0)}$ and $T_{\phi(0)}$, respectively, where $2m$ is the ramification index of $\phi$ at $0$. Let $\widehat{S}$ and $\widehat{T}$ be the minimal resolution of $S\times_B \Gamma$ and $T\times_B \Gamma$, respectively.

The central fiber
$$
\widehat{S}_0 = C_0 \cup C_1 \cup C_2\cup ... \cup C_{2m-1}
$$
of $\widehat{S}/\Gamma$ is a union of $2m$ smooth rational curves of simple normal crossings whose dual graph is a circle, where $C_0$ is the proper transform of $S_{\phi(0)}$ and $C_i\cap C_{i+1} \ne\emptyset$ for $i=0,1,...,2m-1$ with $C_{2m} = C_0$.

It is easy to see that the involution $\sigma_S: S\to S$ lifts to an involution $\widehat{\sigma}_S: \widehat{S} \to \widehat{S}$ whose action on $\widehat{S}_0$ is given by
$$
\begin{aligned}
\widehat{\sigma}_S(C_0\cap C_1) &= C_{2m-1}\cap C_0,\ \widehat{\sigma}_S(C_1\cap C_2) = C_{2m-2}\cap C_{2m-1},\\
\widehat{\sigma}_S(C_2\cap C_3) &= C_{2m-3}\cap C_{2m-2},\ ..., \widehat{\sigma}_S(C_{m-1}\cap C_m) = C_{m}\cap C_{m+1}
\end{aligned}
$$
In the case of $m=1$, $\widehat{\sigma}_S$ switches the two intersections of $C_0$ and $C_1$. The fixed locus $\widehat{\sigma}_S$ consists of four disjoint sections $P_1, P_2, P_3, P_4$ of $\widehat{S}/\Gamma$ with $P_1$ and $P_2$ meeting $C_0$ and $P_3$ and $P_4$ meeting $C_m$.

The exact same holds for $\widehat{T}$:
$$
\widehat{T}_0 = D_0 \cup D_1 \cup D_2\cup ... \cup D_{2m-1}
$$
is a union of $2m$ smooth rational curves of simple normal crossings whose dual graph is a circle and the involution $\sigma_T: T\to T$ lifts to an involution $\widehat{\sigma}_T: \widehat{T} \to \widehat{T}$ whose fixed locus consists of four disjoint sections $Q_1, Q_2, Q_3, Q_4$ of $\widehat{T}/\Gamma$.

Let $\widehat{\sigma} = \widehat{\sigma}_S\times \widehat{\sigma}_T$ be the involution on $\widehat{S}\times_\Gamma \widehat{T}$.
Then the singular locus of $\widehat{S}\times_\Gamma \widehat{T} / \widehat{\sigma}$ consists of the images of the $16$ sections $P_i\times_\Gamma Q_j$ and $4m^2$ isolated points $(C_a\cap C_{a+1}) \times (D_b\cap D_{b+1})$. At each point among 
$$
(C_a\cap C_{a+1}) \times (D_b\cap D_{b+1}),
$$
$\widehat{\sigma}_S\times \widehat{\sigma}_T$ has a $3$-fold rational double point $xy = zw = t$; the same is true for $\widehat{S}\times_\Gamma \widehat{T} / \widehat{\sigma}$ at the images of $(C_a\cap C_{a+1}) \times (D_b\cap D_{b+1})$. So we can easily resolve the singularities of $\widehat{S}\times_\Gamma \widehat{T} / \widehat{\sigma}$ by blowing it up along its singular locus. Let $Z$ be the resulting blowup. Clearly, all components of $Z_0$ are smooth rational surfaces. So we have obtained a resolution of $R\times_B \Gamma$ with the required property via the diagram
$$
\begin{tikzcd}
Z\ar{r} \ar{dr} & \widehat{S}\times_\Gamma \widehat{T} / \widehat{\sigma} \ar{d}\\
& (S\times_B T / \sigma)\times_B \Gamma = R\times_B \Gamma
\end{tikzcd}
$$
\end{proof}

We will outline the monodromy argument. To set this up, suppose that we have a smooth projective family $Z/U$ of Kummer surfaces of maximal moduli over a smooth quasi-projective surface $U$ and a higher Chow class
$$
\xi\in \CH_\QQ^k(\prod_{i=1}^r Y_i \times Z \times \prod_{i=r+2}^m Y_i, 1).
$$
We want to show that $r_{k,1}(\xi_b)$ is trivial for $b\in U$ general.

Given our construction of the one parameter family $S\times_B T$, after a base change, we can find a morphism $C\to U$ from a smooth quasi-projective curve $C$ to $U$ whose image passing through a general point of $U$ with the following property. The one-parameter family $Z_C = Z\times_U C$ over $C$ can be extended to a family $Z_{\overline{C}}$ of Kummer surfaces over the completion $\overline{C}$ of $C$ such that $Z_{\overline{C}}$ is smooth and the pullback $\xi_C$ of $\xi$ to $Z_C$ can be extended to a higher Chow class $\eta_{\overline{C}}\in \CH_\QQ^k(Z_{\overline{C}}, 1)$ satisfying
\begin{equation}\label{HDPEE001}
r_{k,1}(\eta_t) - r_{k,1}(\xi_t)\in \Big(\bigotimes_{i=1}^{r} H_{\mathrm{tr}}^2(Y_i) 
\otimes H_{\mathrm{tr}}^2(Z_t) \otimes H^{4m-2k-2r}(\prod_{i=r+2}^m Y_{i})\Big)^\perp
\end{equation}
for all $t\in C$. Actually, \eqref{HDPEE001} holds for the full regulator $\cl_{k,1}$. That is,
\begin{equation}\label{HDPEE010}
\wcl_{k,1}(\eta_t) - \wcl_{k,1}(\xi_t)\in \Big(\bigotimes_{i=1}^{r} H_{\mathrm{tr}}^2(Y_i) 
\otimes H_{\mathrm{tr}}^2(Z_t) \otimes H^{4m-2k-2r}(\prod_{i=r+2}^m Y_{i})\Big)^\perp
\end{equation}
where $\wcl_{k,1}(\eta_t)$ and $\wcl_{k,1}(\xi_t)$ are local lifts of $\cl_{k,1}(\eta_t)$ and $\cl_{k,1}(\xi_t)$, respectively. The Gauss-Manin connection $\nabla$ on $Y_1\times ...\times Y_r\times Z_C\times Y_{r+2}\times ...\times Y_m/C$ acts on $\wcl_{k,1}(\eta_t)$ and 
$\wcl_{k,1}(\xi_t)$ (see, for example, \cite{CDKL}).

Let us fix a class
\begin{equation}\label{HDPEE003}
\omega \in \bigotimes_{i=1}^{r} H_{\mathrm{tr}}^2(Y_i) \otimes H^{4m-2k-2r}(\prod_{i=r+2}^m Y_{i}).
\end{equation}
Then by \eqref{HDPEE010},
\begin{equation}\label{HDPEE004}
\pi_*(\wcl_{k,1}(\eta_t)\wedge \omega - \wcl_{k,1}(\xi_t)\wedge \omega) \in H_{\mathrm{tr}}^2(Z_t)^\perp
\end{equation}
for all $t\in C$, where $\pi$ is the projection $Y_1\times ...\times Y_r\times Z\times Y_{r+2}\times ...\times Y_m\to Z$.

It follows from \eqref{HDPEE004} that
\begin{equation}\label{HDPEE005}
\nabla \big(\pi_*(\wcl_{k,1}(\eta_t)\wedge \omega - \wcl_{k,1}(\xi_t)\wedge \omega)\big) = 0
\end{equation}
for the Gauss-Manin connection $\nabla$ on $Z_C/C$. Since $\wcl_{k,1}(\eta_t)$ is the restriction of $\wcl_{k,1}(\eta)$ defined on the smooth projective variety $Z_{\overline{C}}$, we have
\begin{equation}\label{HDPEE006}
\nabla \big(\pi_*(\wcl_{k,1}(\eta_t)\wedge \omega)\big) = 0.
\end{equation}
Combining \eqref{HDPEE005} and \eqref{HDPEE006}, we obtain
\begin{equation}\label{HDPEE007}
\nabla \big(\pi_*(\wcl_{k,1}(\xi_t)\wedge \omega)\big) = 0
\end{equation}
on $Z_C/C$.

By our construction of $S\times_B T$, we can choose two such curves $C_i$ with two points $p_i\in C_i$ and maps $f_i:C_i\to U$ for $i=1,2$ satisfying that 
$f_1(p_1) = f_2(p_2) = b$ and the differential maps $df_i$ of $f_i$ on the tangent spaces of $C_i$ at $p_i$ satisfy that
\begin{equation}\label{HDPEE002}
\begin{tikzcd}
T_{C_1,p_1} \oplus T_{C_2,p_2} \ar[->>]{rr}{df_1\oplus df_2} & & T_{U,b}
\end{tikzcd}
\end{equation}
is surjective. By shrinking $U$, let us assume that \eqref{HDPEE002} holds for every $b\in U$.

Then by \eqref{HDPEE002}, we see that \eqref{HDPEE007} actually holds on $Z/U$. Namely,
\begin{equation}\label{HDPEE009}
\nabla \big(\pi_*(\wcl_{k,1}(\xi_b)\wedge \omega)\big) = 0
\end{equation}
on $Z/U$ for the Gauss-Manin connection $\nabla$ on $Z/U$. And since $Z/U$ is a complete family of Kummer surfaces, \eqref{HDPEE009} implies that
\begin{equation}\label{HDPEE008}
\pi_*(\wcl_{k,1}(\xi_b)\wedge \omega) \in H_{\mathrm{tr}}^2(Z_b)^\perp
\end{equation}
for all $b\in U$. And since \eqref{HDPEE008} holds for all $\omega$ in the space \eqref{HDPEE003}, we conclude that
$$
r_{k,1}(\xi_b)\in \Big(\bigotimes_{i=1}^{r} H_{\mathrm{tr}}^2(Y_i) 
\otimes H_{\mathrm{tr}}^2(Z_b) \otimes H^{4m-2k-2r}(\prod_{i=r+2}^m Y_{i})\Big)^\perp
$$
for all $b\in U$.

\section{Bloch-Beilinson Conjecture on Abel-Jacobi Maps}\label{HDPESECBBCAJ}

The following conjecture stated in \cite{Lew1}, can be thought of as
a variant of the Bloch-Beilinson conjecture:

\begin{Conjecture}\label{CONJ030}
Let $V/\ol{\QQ}$ be a
smooth quasiprojective variety. Then the Abel-Jacobi
map $\Phi_{k,\QQ}:\CH^{k}_{\hom}(V/\ol{\QQ};\QQ)
\to J^{k}(V(\CC))\otimes\QQ$ is injective.
\end{Conjecture}

Here the definition of the Abel-Jacobi map for
smooth quasiprojective varieties, which is an extension
of Griffiths' prescription, involves Carlson's
extension class interpretation of intermediate jacobians
(\cite{Ca}). A detailed description of this map 
for example can be
found in \cite[\S 9]{Ja2}.
We now make use of the following result: 

\begin{Theorem}\label{THM031}
{\rm (\cite{Lew1})}
Assume given a smooth  projective variety $X/\CC$.
Then for all $k$, there is a filtration
\[
\begin{split}
\CH^k(X;\QQ) = F^0 &\supset F^1\supset \cdots \supset F^{\ell}
\supset F^{\ell +1}\\
&\supset \cdots \supset F^k\supset F^{k+1}
= F^{k+2}=\cdots,
\end{split}
\]
which satisfies the following

\medskip

{\rm (i)} $F^1 = \CH^k_{\hom}(X;\QQ)$

\medskip

{\rm (ii)} $F^2 \subset \ker \Phi_{k,{\QQ}} : 
\CH^k_{\hom}(X;\QQ) \to J^k(X)\otimes\QQ$.

\medskip

{\rm (iii)} $F^{\ell}\bullet F^r \subset F^{\ell +r}$, where $\bullet$ is
the intersection product.

\medskip

{\rm (iv)} $F^{\ell}$ is preserved under push-forwards $f_{\ast}$ and
pull-backs $f^{\ast}$, where $f : X \to Y$
is a morphism of smooth projective varieties. [In short,
$F^{\ell}$ is preserved under the action of correspondences
between smooth projective varieties.]

\medskip

{\rm (v)} $\Gr_F^{\ell} :=  F^{\ell}/F^{\ell +1}$ factors through the
Grothendieck motive. More specifically,
let us assume that the K\"unneth components of the diagonal
class $[\Delta] = \bigoplus_{p+q = 2n}[\Delta(p,q)] \in
H^{2n}(X\times X,\QQ)$ are algebraic.  Then
$$
\Delta(2n-2k+r,2k-r)_{\ast}\biggl\vert_{\Gr_F^{\ell}\CH^k(X;\QQ)}
= \begin{cases} \text{\rm Identity}&\text{\rm if}\  r = \ell\\
0&\text{\rm otherwise}
\end{cases}
$$

\medskip 

{\rm (vi)} Let $D^k(X) := \bigcap_{\ell}F^{\ell}$. If Conjecture \ref{CONJ030}
above holds, then $D^k(X) = 0$.
\end{Theorem}

Using Theorem \ref{THM031}, it was proved in \cite[Lemma 3.2]{ChenLewis4} that if Conjecture \ref{CONJ030} holds, $\CH_\QQ^2(X\times Y)$ has K\"unneth decomposition
for a product $X\times Y$ of two smooth projective surfaces satisfying $H^1(X) = H^1(Y) = 0$ and
\begin{equation}\label{HDPEE807}
(H^2(X, \QQ)\otimes H^2(Y,\QQ)) \cap H^{2,2}(X\times Y) = H^{1,1}(X, \QQ) \otimes H^{1,1}(Y, \QQ).
\end{equation}
Let us verify \eqref{HDPEE807} for a very general Kummer surface $X$ and a Kummer surface $Y$ with $\rank_\ZZ \Pic(Y)\le 19$. Actually, we have

\begin{Proposition}\label{HDPEPROP001}
Let $\pi: X\to B$ be a non-isotrivial smooth family of K3 surfaces over a smooth variety $B$ and let $Y$ be a smooth K3 surface. Then
$$
(H^2(X_b, \QQ)\otimes H^2(Y,\QQ)) \cap H^{2,2}(X_b\times Y) = H^{1,1}(X_b, \QQ) \otimes H^{1,1}(Y, \QQ)
$$
for $b\in B$ very general.
In particular, the identity \eqref{HDPEE807} holds for the product of a very general Kummer surface and an arbitrary smooth K3 surface.
\end{Proposition}

\begin{proof}
It suffices to prove
\begin{equation}\label{HDPEE808}
(H^{1,1}(X_b, \QQ)^\perp \otimes H^{1,1}(Y, \QQ)^\perp) \cap H^{2,2}(X_b\times Y) = 0
\end{equation}
for $b\in B$ very general,
where $H^{1,1}(X_b, \QQ)^\perp$ and $H^{1,1}(Y, \QQ)^\perp$ are the orthogonal complements of $H^{1,1}(X_b, \QQ)$ and $H^{1,1}(Y, \QQ)$ in $H^2(X_b, \QQ)$ and $H^2(Y, \QQ)$, respectively.

We may take $B$ to be a polydisk and assume that the Kodaira-Spencer map
$$
\begin{tikzcd}
T_{B,b} \ar{r} & H^1(T_{X_b})
\end{tikzcd}
$$
is nonzero at all $b\in B$.

If \eqref{HDPEE808} fails, after shrinking $B$, there exists 
$$
\eta\in H^0(B, (R^2\pi_* \QQ)_{\mathrm{tr}})\otimes H^{1,1}(Y, \QQ)^\perp
$$
such that
$$ \eta_b\ne 0 \in H^{2,2}(X_b\times Y) $$
for all $b\in B$, where $(R^2\pi_* \QQ)_{\mathrm{tr}}$ is the subsheaf of $R^2\pi_* \QQ$
orthogonal to the relative algebraic cycles of $X/B$.

Since $\eta_b$ is orthogonal to
$$
F^1 H^2(X_b) \otimes H^{2,0}(Y) = (H^{1,1}(X_b) \oplus H^{2,0}(X_b)) \otimes H^{2,0}(Y),
$$
we have
$$
\langle \eta, \gamma\otimes \omega_Y\rangle = 0
$$
for all $\gamma\in H^0(B, F^1 R^2\pi_* \CC)$, where $\omega_Y \in H^{2,0}(Y)$ is a nonvanishing holomorphic $2$-form  on $Y$. Applying the Gauss-Manin connection, we obtain
$$
\langle \eta, \nabla \gamma \otimes \omega_Y \rangle = 0
$$
where we observe that $\nabla \eta = 0$. Since the Kodaira-Spencer map of $\pi$ is nonzero, we have
$$
\nabla (F^1 R^2\pi_* \CC) \not\subset F^1 R^2\pi_* \CC \otimes \Omega_B
$$
due to the fact that the pairing $H^{1,1}(X_b) \otimes H^1(T_{X_b}) \to H^{0,2}(X_b)$ is nondegenerate.
Thus, we conclude
$$
\langle \eta_b, \xi_b \otimes \omega_Y \rangle = 0
$$
for all $\xi_b \in H^2(X_b)$ and $b\in B$. That is,
$$
\eta_b \in (H^2(X_b) \otimes H^{2,0}(Y))^\perp.
$$
But we know that
$$
\begin{aligned}
&\quad (H^{1,1}(X_b, \QQ)^\perp \otimes H^{1,1}(Y, \QQ)^\perp) \cap (H^2(X_b) \otimes H^{2,0}(Y))^\perp
\\
&= H^{1,1}(X_b, \QQ)^\perp \otimes (H^{1,1}(Y, \QQ)^\perp \cap H^{1,1}(Y))
\\
&= H^{1,1}(X_b, \QQ)^\perp \otimes (H^{1,1}(Y, \QQ)^\perp \cap H^{1,1}(Y, \QQ)) = 0.
\end{aligned}
$$
This leads to $\eta_b = 0$, which is a contradiction.
\end{proof}

Combining the above proposition and \cite[Lemma 3.2]{ChenLewis4}, we are able to apply Theorem \ref{HDPETHMMAIN} to the case $(k,r,m,n) = (3,1,2,4)$ and conclude that
the Hodge-$\calD$-conjecture fails for the real regulator $r_{3,1}$ on a product of four very general elliptic curves, if the Bloch-Beilinson Conjecture \ref{CONJ030} holds.

\end{document}